\documentclass[12pt]{article}

\usepackage{amsmath,amssymb,amsthm,amsfonts}
\usepackage[numbers]{natbib}
\usepackage[english]{babel}
\usepackage[X2,T1]{fontenc}
\usepackage[utf8]{inputenc} 
\usepackage{enumitem} 
\usepackage[usenames]{xcolor}
\usepackage{tikz-cd}
\usepackage{hyperref}
\usepackage{stmaryrd}
\allowdisplaybreaks
\usepackage{category}
\usepackage[all]{xy}
\usepackage{quiver}
\usepackage{appendix}

\newtheorem{theorem}{Theorem}[section]

\newtheorem{lemma}[theorem]{Lemma}

\newtheorem{corollary}[theorem]{Corollary}

\newtheorem{proposition}[theorem]{Proposition}

\theoremstyle{definition}

\newtheorem{definition}[theorem]{Definition}
\newtheorem{remark}[theorem]{Remark}

\newcommand{\hatt}[1]{{\skew{4} \hat{\rule{0ex}{1.2ex}\smash{#1}}}}

\newcommand{\Fin}{\Set_{\mr f}}
\newcommand{\Hey}{\mb{Hey}_{\mr f}}
\newcommand{\Finc}{\Cat_{\mr f}}
\newcommand{\Fino}{\mb{Cas}_{\mr f}}

\renewcommand{\opDisc}{\mb{opDisc}}

\newcommand{\Ind}{\mr{Ind}}

\newcommand{\HA}{\mb{HA}_{\mr{fp}}}
\newcommand{\HAf}{\mb{HA}_{\mr{f}}}
\newcommand{\HPfp}{\mb{HPt}_{\mr{fp}}}
\newcommand{\HPretopos}{\mb{HPretopos}_{\mr{fp}}}
\newcommand{\HP}{\mb{HPt}_{\mr f}}
\newcommand{\HPt}{\mb{HPt}}
\newcommand{\Pt}{\mb{Pt}_{\mr f}}
\newcommand{\Ptfp}{\mb{Pt}_{\mr{fp}}}

\newcommand{\Por}{\mb{Por}_{\mr{f}}}
\newcommand{\Posf}{\mb{Pos}_{\mr{f}}}

\newcommand{\Catf}{\hatt{\Finc}}
\newcommand{\Casf}{\hatt{\Fino}}
\newcommand{\Casfp}{\hatt{\mb{Cas}}_{\mr f,*}}

\renewcommand{\iso}{^{\!\!\cong}}

\title{Stack Representation of Finitely Presented Heyting Pretoposes I}
\author{Lingyuan Ye}
\date{\today}

\begin{document}
%

%
%

%
%

%
\maketitle              

\begin{abstract}
  This is the first of a series of papers on stack representation of finitely presented Heyting pretoposes. In this paper, we provide the first step by constructing a $(2,1)$-site, which can be thought of as the site of finite Kripke frames, such that the $(2,1)$-category of finitely presented Heyting pretoposes contravariantly embeds into the $(2,1)$-topos of stacks on this $(2,1)$-site. This provides an entry point to use categorical and higher sheaf-theoretic tools to study the properties of certain classes of intuitionistic first-order theories.
\end{abstract}

\tableofcontents

\section{Introduction}\label{sec:intro}

In this paper we are interested in stack representation of finitely presented Heyting pretoposes. We shall comment immediately that by finitely presentable, we mean those Heyting pretoposes which are \emph{complete} w.r.t. \emph{finite models} on \emph{finite Kripke frames}. Typical examples include those equivalent to a finite quotient of the free Heyting pretopos generated by a finite pretopos. A pretopos is \emph{finite} if its \emph{category of models} is essentially finite.

In particular, this notion does \emph{not} coincide with those Heyting pretoposes generated by a first-order intuitionisitic theory with finitely many non-logical symbols and finitely many axioms. In general, such first-order intuitionistic theories will not be complete w.r.t. the particular class of finite models we are interested in. Towards the end of this section, we will comment on the reason why we are interested in this restricted class of Heyting pretoposes.

Categorically, this work can be thought of as a \emph{categorification} of the sheaf representation of the category of finitely presented Heyting algebras described in the book~\cite{ghilardi2013sheaves}. We give a brief overview below.

The starting point of the sheaf representation of finitely presented Heyting algebras is the \emph{duality} of finite distributive lattices and finite Heyting algebras as follows,
\[
  \begin{tikzcd}
    \mb{HA}\op_{\mr f} \ar[d, tail] \ar[r, "\simeq"] & \Por \ar[d, tail] \\
    \DL\op_{\mr f} \ar[r, "\simeq"'] & \Posf 
  \end{tikzcd}
\]
where $\DL_{\mr f}$, $\mb{HA}_{\mr f}$ are the categories of finite distributive lattices and finite Heyting algebras, respectively. Notice that $\mb{HA}_{\mr f}$ is a wide subcategory of $\DL_{\mr f}$, since any finite distributive lattice is automatically Heyting. Their corresponding dual categories are $\Posf$, the category of finite posets, and $\Por$, the wide subcategory of $\Posf$ with \emph{open} functions, or \emph{p-morphisms} as denoted by modal logicians.

Concretely, the duality between finite posets and finite distributive lattices takes any poset $L$ to the distributive lattice $D(L)$ of upward closed subsets of $L$, while $L$ will be isomorphic to the poset of \emph{models} of $D(L)$. Since any finite distributive lattice is also a Heyting algebra, for any finitely presented Heyting algebra $H$ we may look the poset of its $L$-valued models $\HA(H,D(L))$ for a finite poset $L$. In particular, a Heyting algebra morphism from $H$ to $D(L)$ is exactly the same as a Krpike-model on $L$ for the intuitionistic propositional theory encoded by $H$.

This gives us a \emph{presheaf} on $\Por$,
\[ \HA(H,D(-)) : \Por\op \to \Set, \]
which turns out to be a sheaf for a suitable Grothendieck topology $J$ on $\Por$. The topology $J$ on $\Por$ is in fact the \emph{canonical} one, with covering families being jointly surjective maps. This gives us a sheaf representation functor 
\[ \Phi_- : \HA\op \inj \sh(\Por,J), \]
which intuitively embeds any finitely presented Heyting algebra to the sheaf of its models on finite Kripke frames.

As is well-known, categorification is an art rather than a science. But if we are careful enough, there will still be tracible links. Recall the indexation for higher categories, e.g. see~\cite{nlab:n-category}. For $r,s\in\ov{\N} = \N \cup \set{\infty}$, an $(r,s)$-category is a higher category where $k$-cells are trivial for all $k>r$ and $j$-cells are inverible for all $j > s$. An $(n,n)$-category in this sense is the same as an $n$-category; an $(n,0)$-category is usually called an $n$-groupoid; while an $(n,n+1)$-category can be identified with an $n$-poset.

In this way, Heyting algebras can be thought of as $(0,1)$-Heyting pretoposes, and in full generality the collection of Heyting algebras should organise themselves into a $(1,2)$-category. However, the ordinary sheaf representation chooses a $1$-site $\Por$, which is just an ordinary category, and embeds the underlying $1$-category $\HA$ of finitely presented Heying algebras into the category of $1$-sheaves on this $1$-site. The underlying $1$-category $\HA$ actually suffices for the purpose of studying propositional intuitionistic logic. For instance, the interpolation property and Beth definability for propositional intuitionistic logic are concerned with \emph{pushout} squares in $\HA$, e.g. see~\cite{pitts1983heytingalg}, rather than lax or oplax colimits computed in the $(1,2)$-category of Heyting algebras.

Now the categorification we are working towards concerns increasing the first entry. We intend to study the collection of $(1,1)$-Heyting pretoposes, which are simply Heyting pretoposes in the usual 1-categorical sense. They organise themselves into a $2$-category, while we are concerned in providing a stack representation for its underlying $(2,1)$-category, i.e. the $2$-category of Heyting pretoposes with Heyting functors and natural \emph{isomorphisms} between them. We denote the $(2,1)$-category of finitely presented Heyting pretoposes as $\HPfp$, while the full $2$-category as $\HPretopos$. Similar convention applies to $\Ptfp$ and $\Pretopos_{\mr{fp}}$ for pretoposes, and their finitary versions. Again, for the purpose of studying the properties of first-order intuitionistic logic, $\HPfp$, rather than $\HPretopos$, is the right environment we want to work in. For instance, as shown in~\cite{pitts1983open}, the interpolation property of first-order intuitionistic logic is again concerned with $2$-pushouts in $\mb{HPt}$, rather than cocomma squares in the full $2$-category $\mb{HPretopos}$ of Heyting pretoposes.

Thus, our first task in Section~\ref{sec:finitepretopos} is to obtain a categorification of the duality result for \emph{finite pretoposes} and \emph{finite Heyting pretoposes}. And it turns out that, according to our definition, finite pretoposes are again automatically \emph{Heyting}, in fact even \emph{toposes}.\footnote{In this paper, we will follow the more traditional usage of the notion of topos to mean \emph{elementary topos}, unless otherwise stated.} And a similar duality result holds as indicated below,
\[
  \begin{tikzcd}
    \mb{HPretopos}_{\mr f}\op \ar[d, tail] \ar[r, "\simeq"] & \Casf \ar[d, tail] \\
    \Pretopos_{\mr f}\op \ar[r, "\simeq"'] & \Catf
  \end{tikzcd}
\]
Here $\Catf$ is the $2$-category of \emph{finite Cauchy complete categories}, and $\Casf$ is its wide subcategory with a suitable notion of \emph{open} functors. This gives us a duality for the full 2-categorical version, and for us latter we will restrict to their underlying $(2,1)$-categories to construct a $(2,1)$-site.

After obtaining the duality, in Section~\ref{sec:2-topos} we proceed to construct the stack representation for the $(2,1)$-category $\HPfp$. We similarly choose the underlying $(2,1)$-category to be $\Casf\iso$, the underlying $(2,1)$-category of $\Casf$ which is dual to $\HP$. By putting a suitable Grothendieck topology $J$ on $\Casf\iso$, we once more obtain a stack representation $2$-functor
\[ \Phi_- : \HPfp\op \inj \St(\Casf\iso,J). \]
In this paper, our use of the term stack follows the traditional convention, i.e. it indicates suitable \emph{groupoid} valued pseudo-functors. In other words, we have embedded the opposite $(2,1)$-category of $\HPfp$ into the $(2,1)$-category of $(2,1)$-sheaves on the $(2,1)$-site $(\Casf,J)$.

We will denote the above representation functor as the \emph{gros} representation functor. It also generates a \emph{petit} representation functor for a single finitely presented Heyting pretopos $\mc H$, via the embedding
\[ \mc H \inj \mc H/\HPfp\op, \]
sending each $\varphi\in\mc H$ to the pullback onto the slice category over it $\mc H \to \mc H/\varphi$. Composed with the gros representation, this gives us a representation of $\mc H$ as follows,
\[ \mc H \inj \St(\Casf\iso,J)/\Phi_{\mc H}. \]
Section~\Ref{subsec:grospetit} will be devoted to the study of this induced petit representation. In fact, there we will show that the image of $\mc H$ under the petit representation restricts to \emph{discrete opfibrations} over $\Phi_{\mc H}$ in the $(2,1)$-category $\St(\Casf\iso,J)$, and the induced functor
\[ \mc H \inj \opDisc(\Phi_{\mc H}) \]
will be a \emph{Heyting} functor. In this sense, the internal structure of the $(2,1)$-topos $\St(\Casf\iso,J)$ realises the Heyting structure of \emph{every} finitely presented Heyting pretoposes.

One of the reasons to develop the ordinary sheaf representation of Heyting algebras is that, it can be used to show that $\HA\op$ is itself a \emph{Heyting category}, by proving $\HA\op$ is closed under the Heyting operation in $\sh(\Por,J)$. As explained in~\cite{ghilardi2013sheaves}, this implies that propositional intuitionisitic logic has quantifier elimination for \emph{second-order} quantifiers which is originally proven in~\cite{pitts1992uniform}, and that the algebraic theory of Heyting algebras has a model companion.

Similarly, we expect our stack representation for finitely presented Heyting algebras to show that $\HPfp\op$ will also be a $(2,1)$-Heyting category. For the notion of higher Heyting categories, see~\cite{nlab:heyting_2-category}. In particular, we expect to show that $\HPfp\op$ via this stack representation is closed under the $(2,1)$-Heyting operations in $\St(\Casf\iso,J)$.

The first step towards such a result is to characterise the image of the stack representation of $\HPfp\op$ in $\St(\Casf\iso,J)$. In the sequal to this paper, we will characterise the image of the stack representation functor in $\St(\Casf\iso,J)$, by introducing a suitable notion of \emph{bisimulation} between models of finitely presented Heyting pretoposes. We will also show in the sequal that the topology $J$ on $\Casf\iso$ exactly encodes this logical information of bisimulation, and the notion of $J$-stacks is equivalent to a bisimulation-invariant class of models. Under this notion of bisimulation, we are also going to characterise the image of the stack representation functor as certain \emph{finitary} stacks.

Along with such semantic investigations, we also expect to extract similar quantifier elimination results of \emph{higher quantifiers} for this restricted class of first-order intuitionistic theories in future work. This also explains our motivation for looking at first-order intuitionistic theories with strong finite model properties: For a general first-order intuitionistic theory, if it contains infinite models, then the second-order extension will in general have \emph{stronger} expressive power than the first-order fragment. However, for the restricted class of Heyting pretoposes, such a result is at least not outright contradictory. And from our limited knowledge, there hasn't been an attempt for such a result in the literature.

\section{Duality for Finite Pretoposes}\label{sec:finitepretopos}

As mentioned in the introduction, our first objective is to study a suitable notion of finite pretoposes. For us, the definition that works is the following:

\begin{definition}
  A pretopos is \emph{finite} iff its category of models is essentially finite, i.e. equivalent to a finite category.
\end{definition}

It is then natural to ask whether all finite categories could be the category of models of some pretopos. The answer is negative due to the following simple observation:

\begin{lemma}\label{lem:modcaucomp}
  For any pretopos $\mc D$, its category of models $\Mod(\mc D)$ is closed under filtered colimits in $[\mc D,\Set]$, thus in particular is Cauchy complete.
\end{lemma}
\begin{proof}
  The fact that $\Mod(\mc D)$ is closed under filtered colimits computed in $[\mc D,\Set]$ is well-known, e.g. see Chapter 5 of~\cite{adamek1994locally}. The splitting of idempotent is in particular a filtered colimit, thus $\Mod(\mc D)$ must be Cauchy complete.
\end{proof}

In fact, for finite categories, Cauchy completeness is the only obstacle, i.e. all Cauchy complete finite categories can arise as the category of models of a finite theory. Moreover, we obtain a duality result, where there is an equivalence
\[ \Catf\op \simeq \Pretopos_{\mr f}, \]
from the opposite 2-category of all Cauchy complete finite categories, to the 2-category of finite pretoposes. We prove this by proving the following duality
\[ \Pretopos_{\mr f}\op \simeq \Topos_{\mr f}, \]
where $\Topos_{\mr f}$ is the 2-category of $\Fin$-bounded toposes. The previous duality follows from this because from Corollary 2.2.23 in~\cite{johnstone2002sketches} we already know that the 2-category of finite Cauchy complete categories is equivalent to the 2-category of $\Fin$-bounded toposes. This latter result can be seen as a consequence of the following facts:
\begin{itemize}
\item There is a general 2-fully faithful embedding of Cauchy complete categories into toposes with \emph{essential} geometric morphism,
  \[ \hatt{\Cat} \hook \mb{EssTopos}. \]
\item In~\cite{Haigh1980}, it has shown that any geometric morphism between $\Fin$-bounded toposes is \emph{automatically essential}.
\item It can be shown that any $\Fin$-bounded topos is of the form $\Fin[\mc I]$ for some finite Cauchy complete category $\mc I$, by showing that any Grothendieck topology on finite Cauchy complete categories are so called \emph{rigid}, thus is of presheaf type. Again, see Lemma 2.2.21 in~\cite{johnstone2002sketches}.
\end{itemize}

Section~\Ref{subsec:fincat} will first provide an analysis of finite pretoposes, while Section~\Ref{subsec:dualityfinitepretopos} proceeds to prove the above mentioned duality result for finite pretoposes. By the equivalence $\Catf \simeq \Topos_{\mr f}$, we also study certain geometric aspects of finite Cauchy complete categories, notably that of geometric surjections, in Section~\Ref{subsec:surjection}, which turns out to be crucial for the Grothendieck topology we will be constructing in the next section. Finally in Section~\Ref{subsec:openheyting}, we restrict the duality for finite pretoposes to one for finite Heyting pretoposes, and study certain properties of open functors, again in preparation for the next section.

\subsection{Finite Categories as Models}\label{subsec:fincat}

To show every Cauchy complete finite category can arise as the category of models of some theory, the following observation is of vital importance:

\begin{lemma}\label{lem:fincatcaueqind}
  For any finite category $\mc I$, we have a canonical equivalence
  \[ \hatt{\mc I} \simeq \Ind(\mc I), \]
  where $\hatt{\mc I}$ is the \emph{Cauchy completion} of $\mc I$, and $\Ind(\mc I)$ is its \emph{ind-completion}.
\end{lemma}
\begin{proof}
  This follows from the fact that for a finite category, essentially the only non-trivial filtered colimit is the splitting of idempotency, thus $\hatt{\mc I}$ already has all filtered colimits. For a proof, see~\cite{giacomo2022accessible}.
\end{proof}

In particular, Lemma~\Ref{lem:fincatcaueqind} implies that if $\mc I$ is already Cauchy complete, then $\mc I \simeq \Ind(\mc I)$. Lemma~\Ref{lem:fincatcaueqind} also allows us to compute the category of points of the copresheaf category on a finite category:

\begin{lemma}\label{lem:finmodelcat}
  For any finite category $\mc I$, there is an equivalence
  \[ \Topos(\Set,[\mc I,\Set]) \simeq \hatt{\mc I}. \]
\end{lemma}
\begin{proof}
  We have the following sequence of equivalences,
  \[ \Topos(\Set,[\mc I,\Set]) \simeq \mr{Flat}(\mc I\op,\Set) \simeq \Ind(\mc I) \simeq \hatt{\mc I}. \]
  The first equivalence holds by Diaconescu's theorem; the second equivalence is well-known; and the last equivalence is by Lemma~\Ref{lem:fincatcaueqind}.
\end{proof}

In particular, this gives us a very simple way of determining what the classifying pretopos should be for finite categories: Let us write $\Fin[\mc I]$ for the functor category $[\mc I,\Fin]$,

\begin{proposition}
  For any finite category $\mc I$, we have an equivalence
  \[ \Fin[\mc I] \simeq \Fin[\hatt{\mc I}], \]
  such that
  \[ \Mod(\Fin[\mc I]) \simeq \Mod(\Fin[\hatt{\mc I}]) \simeq \hatt{\mc I}. \]
\end{proposition}
\begin{proof}
  It is easy to see that $\Fin[\mc I]$ is exactly the coherent objects in the Grothendieck topos $[\mc I,\Set]$ because $\mc I$ is a finite category. Now we have a well-known equivalence
  \[ [\mc I,\Set] \simeq [\hatt{\mc I},\Set], \]
  thus the equivalence restricts to their corresponding coherent objects
  \[ \Fin[\mc I] \simeq \Fin[\hatt{\mc I}]. \]
  Furthermore, since the copresheaf topos $[\mc I,\Set]$ is coherent, its category of points is equivalent to the category of models of the pretopos $\Fin[\mc I]$, thus we have
  \[ \Mod(\Fin[\mc I]) \simeq \Topos(\Set,[\mc I,\Set]) \simeq \hatt{\mc I}. \qedhere \]
\end{proof}

Thus for any category $\mc I$, it embeds into $\Mod(\Fin[\mc I]) \simeq \hatt{\mc I}$ as follows: For any $x\in\mc I$, as a model it is simply the evaluation functor
\[ x^{*} : \Fin[\mc I] \to \Fin, \]
which we view as precomposition with $x : \mb 1 \to \mc I$ where $\mb 1$ is the terminal category. Since finite limits and colimits are computed point-wise in $\Fin[\mc I]$, $x^{*}$ preserves all finite limits and colimits, thus is in particular coherent. Given $f : x \to y$, this gives a natural transformation
\[ f : x \to y : \mb 1 \to \mc I, \]
and by wiskering with $f$ we get a natural transformation
\[ f^{*} : x^{*} \to y^{*}. \]

Furthermore, we can thus show that all finite pretoposes are equivalent to this form:

\begin{corollary}\label{cor:finpretoposfincaucat}
  Any finite pretopos is equivalent to one of the form $\Fin[\mc I]$ for some finite Cauchy complete category $\mc I$. In particular, all finite pretoposes are Heyting, and in fact toposes.
\end{corollary}
\begin{proof}
  This follows from Makkai's conceptual completeness theorem~\cite{makkai1982stone}. Let $\mc D$ be any finite pretopos with category of models equivalent to some finite Cauchy complete category $\mc I$. Then we have an equivalence between category of models as follows,
  \[ \Mod(\mc D) \simeq \mc I \simeq \Mod(\Fin[\mc I]), \]
  which by conceptual completeness provides a syntactic equivalence
  \[ \mc D \simeq \Fin[\mc I]. \]
  Thus all finite pretoposes are of this form.
\end{proof}

The above result also essentially appears as exercises in SGA4, see Exercise 3.12 in volume 2 of~\cite{SGA4}. There a Grothendieck topos is called \emph{finite} if it is equivalent to the category of presheaves on a finite category. This is consistent with our definition of a finite pretopos, in that Corollary~\Ref{cor:finpretoposfincaucat} implies a pretopos is finite in our sense iff its classifying topos is finite in the sense of SGA4.

\begin{remark}
  From a logical perspective, the pretopos $\Fin[\mc I]$ is the classifying theory for flat functors out of $\mc I\op$, which can be written down quite explicitly, e.g. see Chapter 2 of~\cite{caramello2018theories}. This is somewhat surprising in that all coherent first-order theories whose category of models is essentially finite can be axiomatised in a \emph{uniform} way. 
\end{remark}

For any pretopos $\mc D$ and any model of $\mc D$
\[ M : \mc D \to \Set, \]
there is also the notion of this particular model $M$ being finite. We say $M$ is \emph{finite} if it factors through $\Fin$. For a finite pretopos, not only its category of models is essentially finite, but all of its models are also finite:

\begin{lemma}\label{lem:finpretoposfinmodel}
  All models of a finite pretopos are finite.
\end{lemma}
\begin{proof}
  By Corollary~\ref{cor:finpretoposfincaucat}, it suffices to consider those toposes of the form $\Fin[\mc I]$. Now notice that $\Fin[\mc I]$ is generated by a \emph{countable}, in fact finite, theory, thus by L\"owenheim–Skolem theorem, if it has an infinite model then it must have a model of each cardinality, which implies $\Mod(\Fin[\mc I])$ cannot be essentially finite. Thus, any model $M : \Fin[\mc I] \to \Set$ must be finite.
\end{proof}

One may wonder whether the converse of Lemma~\Ref{lem:finpretoposfinmodel} is also true, i.e. if a pretopos only has finite models, then whether or not it must be finite as well. The answer to this question is \emph{negative}. Consider the theory $\Fin[\Z] = [\Z,\Fin]$ where $\Z$ is the integer group considered as a category with one element. It can be shown that $\Fin[\Z]$ is \emph{categorical}, i.e. it only has one model upto isomorphism, and in fact we have an equivalence
\[ \Mod(\Fin[\Z]) \simeq \pf{\Z}, \]
where $\pf\Z$ is the profinite completion of $\Z$, which is a profinite groupoid with one connected component. A canonical choice of a model of $\Fin[\Z]$ is given by the forgetful functor
\[ \Fin[\Z] \to \Set, \]
which factors through $\Fin$, thus is finite. This implies that all models of $\Fin[\Z]$ are finite, because all of them are isomorphic to a finite model. However, the profinite groupoid $\pf\Z$ is by no means finite, because its morphism group is uncountable. More generally, for any group $G$, we have the equivalence
\[ \Mod(\Fin[G]) \simeq \pf G, \]
where $\pf G$ is the profinite completion of $G$. And all of the theories of such form have only finite models, but their category of models is not necessarily essentially finite. For a reference for these calculations see~\cite{Rogers2023toposesof}.

From a logical perspective, the caveat is that a theory might have \emph{countably many sorts}. For instance, to present $\Fin[\Z]$ as the classifying pretopos of a certain coherent theory, we must have countably many sorts, because there are countably many connected components in $\Fin[\Z]$. Now even though any model of $\Fin[\Z]$ must be finite, there is no bound on the size of the interpretation of the sorts, i.e. for any $n\in\N$, there exists a sort of $\Fin[\Z]$ such that its size in this model is larger than $n$. Thus, even though this model is finite, the automorphism group of this model can still be uncountable.

\subsection{Duality for Finite Pretoposes}\label{subsec:dualityfinitepretopos}

Now that we have understood better the class of finite pretoposes, we proceed further to establish a duality result. The first task is to describe the functoriality of the construction $\mc I \mapsto \Fin[\mc I]$. For any functor $A : \mc I \to \mc J$ with $\mc I,\mc J$ finite Cauchy complete categories, by precomposition we have a functor
\[ A^{*} : \Fin[\mc J] \to \Fin[\mc I], \]
which is coherent because the coherent structure on $\Fin[\mc J]$ and $\Fin[\mc I]$ are computed point-wise. Given a natural transformation
\[ \eta : A \to B : \mc I \to \mc J, \]
there is an induced natural transformation
\[ \eta^{*} : A^{*} \to B^{*} : \Fin[\mc J] \to \Fin[\mc I], \]
by wiskering, such that for any $F : \mc J \to \Fin$,
\[ \eta^{*}_F = F\eta : FA \to FB \]
This thus gives us a 2-functor from the 2-category of finite categories to the 2-category of finite pretoposes,
\[ \Fin[-] : \Catf\op \to \Pretopos_{\mr f}. \]
In fact, we have the following diagramme
\[
  \xymatrix{
    \Catf \ar[rr] \ar[dr]_{\simeq} & & \Pretopos_{\mr f}\op \\
    & \Topos_{\mr f} \ar[ur] &
  }
\]
where as mentioned previously, under the same construction the 2-functor
\[ \Fin[-] : \Catf \simeq \Topos_{\mr f} \]
is an equivalence. Hence, we only need to show $\Topos_{\mr f} \to \Pretopos_{\mr f}\op$, by sending a geometric morphism to its inverse image, is also an equivalence.

Notice that by Corollary~\Ref{cor:finpretoposfincaucat}, the 2-functor $\Catf \to \Pretopos_{\mr f}\op$ is essentially surjective, thus so is $\Topos_{\mr f} \to \Pretopos_{\mr f}\op$. To show it is an equivalence, we prove it is also 2-fully faithful, which amounts to prove that any coherent functor between finite pretoposes are in fact the inverse image part of a geometric morphism. Hence, the crucial part is to show any coherent functor between $\Fin[\mc J]$ and $\Fin[\mc I]$ indeed preserves all finite \emph{colimits}.

We first compute the subobject lattices in $\Fin[\mc J]$. For any finite category $\mc I$, we will use $D(\mc I)$ to denote the finite distributive lattice of all upward closed subsets of $\mc I$. Given a subset $U$ of objects of $\mc I$, it is upward closed if for any $x\in U$ and $f : x \to y$ in $\mc I$, we also have $y\in U$. Using this notion, it is easy to see that:

\begin{lemma}\label{lem:subobjfin}
  For any finite category $\mc J$ and $X\in\Fin[\mc J]$, there is an isomorphism
  \[ \sub_{\Fin[\mc J]}(X) \cong D(\elem_{\mc J}X). \]
  where $\elem_{\mc J}X$ is the category of elements of $X$.
\end{lemma}
\begin{proof}
  We first show the case when $X = 1$ is the terminal object in $\Fin[\mc J]$. In this case, given any subobject $\sigma \inj 1$, there is an associated upward closed subset $U_{\sigma}$, such that for any $x\in\mc I$,
  \[ x \in U_{\sigma} \eff \sigma(x) \cong 1. \]
  It is straight forward to verify that this gives us an isomorphism. For general $X$, we have the following sequence of isomorphisms,
  \[ \sub_{\Fin[\mc J]}(X) \cong \sub_{\Fin[\mc J]/X}(1) \cong \sub_{\Fin[\elem_{\mc J}X]}(1) \cong D(\elem_{\mc J}X). \]
  This completes the proof.
\end{proof}

In particular, for any finite category $\mc J$, Lemma~\Ref{lem:subobjfin} implies that the subobject lattice on $X$ for any $X\in\Fin[\mc J]$ is a finite distributive lattice, since $\elem_{\mc J}X$ is a finite category. This allows us to show:

\begin{lemma}\label{lem:cohprecolim}
  Any coherent functor $M : \Fin[\mc J] \to \Fin[\mc I]$ preserves all finite colimits.
\end{lemma}
\begin{proof}
  It suffices to show $M$ preserves all coequalisers. Given any coequaliser, taking a countable union of subobjects can inductively generate an effective coequaliser that has the same colimit. Since the subobject lattices in $\Fin[\mc J]$ are \emph{finite} distributive lattices, such countable unions are in fact finite, thus are preserved under coherent functors. Of course effective quotients are also preserved by coherent functors, thus they preserve arbitrary coequalisers.
\end{proof}

\begin{lemma}\label{lem:cohfinessgeom}
  Any coherent functor from $\Fin[\mc J]$ to $\Fin[\mc I]$ is the inverse image part of a geometric morphism.
\end{lemma}
\begin{proof}
  Consider any coherent functor $A^{*} : \Fin[\mc J] \to \Fin[\mc I]$, we construct its right adjoint $A_{*}$ as follows: For any $X\in\Fin[\mc I]$ and $a\in\mc J$,
  \[ A_{*}X(a) \cong \Fin[\mc I](A^{*}\yon^a,X), \]
  where $\yon^a$ is the corepresentable on $a\in\mc J$. We show that this indeed defines the right adjoint of $A^{*}$: For any $Y\in\Fin[\mc J]$ and $X \in \Fin[\mc I]$,
  \begin{align*}
    \Fin[\mc J](Y,A_{*}X)
    &\cong \Fin[\mc J](\ct{\yon^a \to Y}\yon^a,A_{*}X) \\
    &\cong \lt{\yon^a\to Y}\Fin[\mc J](\yon^a,A_{*}X) \\
    &\cong \lt{\yon^a\to Y}\Fin[\mc I](A^{*}\yon^a,X) \\
    &\cong \Fin[\mc I](\ct{\yon^a\to Y}A^{*}\yon^a,X) \\
    &\cong \Fin[\mc I](A^{*}Y,X)
  \end{align*}
  The third isomorphism holds due to our definition of $A_{*}X$. Also notice that the last step uses the fact that $A^{*}$ preserves finite colimits.
\end{proof}

As a corollary, we obtain our desired duality result between finite pretoposes and finite Cauchy complete categories:

\begin{theorem}\label{thm:finpredual}
  There are equivalences of 2-categories as follows
  \[ \Catf \simeq \Topos_{\mr f} \simeq \Pretopos_{\mr f}\op, \]
  under the construction $\Fin[-]$.
\end{theorem}
\begin{proof}
  As mentioned, we only need to show that the following functor
  \[ \Topos_{\mr f} \to \Pretopos_{\mr f}\op \]
  by taking a geometric morphism to its inverse image is an equivalence. This is a corollary of Lemma~\Ref{lem:cohfinessgeom}.
\end{proof}

One nice thing about showing the above equivalence is that we can now also compute explicitly what is the left adjoint $A_{!}$ of any coherent functor between two finite pretoposes, since now we know that any such coherent functor $A^{*} : \Fin[\mc J] \to \Fin[\mc I]$ is induced by some functor $A : \mc I \to \mc J$. Since $A_{!}$ is a left adjoint, it suffices to know its value on corepresentables:

\begin{lemma}\label{lem:essenleftyon}
  For any $A : \mc I \to \mc J$, for any $x\in\mc I$ we have
  \[ A_{!}\yon^x \cong \yon^{Ax}. \]
\end{lemma}
\begin{proof}
  For any $X\in\Fin[\mc J]$, there is the following natural isomorphisms
  \[ \Fin[\mc J](A_{!}\yon^x,X) \cong \Fin[\mc I](\yon^x,A^{*}X) \cong X(Ax) \cong \Fin[\mc J](\yon^{Ax},X). \]
  This implies the two copresheaves are naturally isomorphic.
\end{proof}

The previously shown equivalence is indeed a categorification of the $(0,1)$-level duality between finite posets and finite distributive lattices in the following sense:

\begin{proposition}\label{prop:finboollocref}
  There is a commutative diagramme as follows,
  \[
    \begin{tikzcd}
      \Posf\op & \DL_{\mr f} \\
      \Catf\op & \mb{Pretopos}_{\mr f}
      \arrow[hook, curve={height=-6pt}, from=1-1, to=2-1]
      \arrow[hook, curve={height=-6pt}, from=1-2, to=2-2]
      \arrow["\simeq"', from=2-1, to=2-2]
      \arrow["\simeq", from=1-1, to=1-2]
      \arrow[curve={height=-6pt}, from=2-1, to=1-1]
      \arrow["\dashv"{anchor=center}, draw=none, from=2-1, to=1-1]
      \arrow[curve={height=-6pt}, from=2-2, to=1-2]
      \arrow["\dashv"{anchor=center}, draw=none, from=2-2, to=1-2]
    \end{tikzcd}
  \]
  where the vertical left adjoints are posetal reflection of categories and localic reflection of pretoposes, respectively.
\end{proposition}
\begin{proof}
  For a finite poset $I$, its dual is the distributive lattice $D(I)$ of upward closed subsets of $I$. For the two equivalences to commute with the two right adjoint inclusions, we need to show that
  \[ \sh_{\mr f}(D(I)) \simeq \Fin[I], \]
  i.e. the category of finite sheaves on $D(I)$ is equivalent to the presheaf finite topos $\Fin[I]$. But this is evident, because any element in $D(I)$ is a \emph{finite} join of corepresentable elements, thus the sheaf condition implies that the value of a sheaf on $D(I)$ is completely determined by that in $I$. This means that the equivalences commutes with the right adjoint inclusions, thus they also commutes with the left adjoints, since they are equivalences, and adjoints are unique upto natural isomorphism.
\end{proof}

Another way to look at the above proof if to use our result in Lemma~\ref{cor:finpretoposfincaucat}, which suggests that the free pretopos generated by $D(I)$ is indeed given by $\Fin[I]$, because
\[ \Mod(\Fin[I]) \simeq I \simeq \Mod(D(I)). \]
The above reflection then also implies that, for any Cauchy complete finite category $\mc I$, we can compute its localic reflection by computing the posetal reflection of $\mc I$. In particular, let $I$ denote its posetal reflection, then we have
\[ D(\mc I) \cong \sub_{\Fin[\mc I]}(1) \cong \sub_{\Fin[I]}(1) \cong D(I). \]

\begin{remark}
  The above equivalence can also be restricted to the Boolean case. It is also well-known that a topos of the form $\Fin[\mc I]$ is \emph{Boolean} iff $\mc I$ is a groupoid. In particular, if $\mc I$ is a poset, then it must be discrete. Thus from above we furthermore obtain a duality for finite Boolean pretoposes and finite Boolean algebras as follows,
  \[
    \begin{tikzcd}
      \Fin\op & \BA_{\mr f} \\
      \Grpd_{\mr f}\op & \mb{BPretopos}_{\mr f}
      \arrow[hook, curve={height=-6pt}, from=1-1, to=2-1]
      \arrow[hook, curve={height=-6pt}, from=1-2, to=2-2]
      \arrow["\simeq"', from=2-1, to=2-2]
      \arrow["\simeq", from=1-1, to=1-2]
      \arrow[curve={height=-6pt}, from=2-1, to=1-1]
      \arrow["\dashv"{anchor=center}, draw=none, from=2-1, to=1-1]
      \arrow[curve={height=-6pt}, from=2-2, to=1-2]
      \arrow["\dashv"{anchor=center}, draw=none, from=2-2, to=1-2]
    \end{tikzcd}
  \]
\end{remark}

\subsection{Geometric Surjection}\label{subsec:surjection}

To restrict the duality $\Catf \simeq \Pretopos_{\mr f}\op$ to one for finite \emph{Heyting} pretoposes, we need a convenient geometric language on the dual side of finite Cauchy complete categories. Recall the following equivalence
\[ \Catf \simeq \Topos_{\mr f}. \]
This way, a lot of geometric properties in the domain of toposes can be transferred to the world of finite Cauchy complete categories. For any property $P$ of geometric morphisms in topos theory, we say a functor $A : \mc I \to \mc J$ in $\Catf$ is $P$, iff its induced geometric morphism $A : \Fin[\mc I] \to \Fin[\mc J]$ is $P$.

For us, the most crucial geometric notion is that of a \emph{surjection}. Recall a geometric morphism $A$ is a surjection iff its inverse image $A^{*}$ is faithful, and we would like to characterise surjections by viewing $A$ as functors $A : \mc I \to \mc J$. For this we use the extensive results in~\cite{caramello2019denseness}:

\begin{lemma}\label{lem:surjessretra}
  $A : \mc I \to \mc J$ is a surjection iff it is essentially surjective upto retracts, i.e. for any $a\in\mc J$, it is a retract in $\mc J$ of $Ax$ for some $x\in\mc I$.
\end{lemma}
\begin{proof}
  Note that the geometric morphism $A : \Fin[\mc I] \to \Fin[\mc J]$ is induced by viewing the opposite functor $A\op : \mc I\op \to \mc J\op$ as a \emph{comorphism of site}. By Proposition 7.1 in~\cite{caramello2019denseness}, since $A\op$ preserves covers, $A$ is a surjection iff $A\op$ is \emph{dense}, iff for any $a\in\mc J$, there exists some $x\in\mc I$ and an arrow $s : a \to Ax$ such that $1_a$ factors through $s$. This is equivalent to say that there exists $r : Ax \to a$ with $rs = 1_a$, i.e. $A$ is essentially surjective upto retracts.
\end{proof}

Recall that we have an orthogonal factorisation system in $\Topos_{\mr f}$ consisting of geometric inclusions and geometric surjections. By Lemma~\Ref{lem:surjessretra}, this also translates to the following orthogonal factorisations system $(E,M)$ on $\Catf$:
\begin{align*}
  E &= \text{essentially surjective upto retracts},\\
  M &= \text{fully faithful whose image is closed under retracts} 
\end{align*}
It is well-known that these form an orthogonal factorisation system on $\Cat$, e.g. see~\cite{dupont2003proper}. It also restricts one to $\Catf$, since for any functor $A : \mc I \to \mc J$ in $\Catf$, the factorisation is given by $\mc I \to \hatt{\im A} \to \mc J$, where $\hatt{\im A}$ is the retract closure of $\im A$ in $\mc J$, which is again Cauchy complete. However, the nice thing in $\Catf$ is that the image of a full functor is automatically closed under retract:

\begin{lemma}\label{lem:fullcloseretr}
  Let $A : \mc I \to \mc J$ be a full functor in $\Catf$. Then the image of $A$ is closed under retracts in $\mc J$, i.e. for any retract $r : Ax \gb a : s$ with $x\in\mc I$ and $a\in\mc J$ there exists some $y\in\mc I$ that $Ay \cong a$.
\end{lemma}
\begin{proof}
  Suppose we have a retract pair in $\mc J$
  \[ r : Ax \gb a : s. \]
  Let $g = sr : Ax \to Ax$, and let $\mc I_g(x,x)$ denote the set of endomorphisms on $x$ in $\mc I$ that is mapped to $g$. By fullness, $\mc I_g(x,x)$ is non-empty. Since $g$ is idempotent, we know that $\mc I_g(x,x)$ is a multiplicative closed set in $\mc I(x,x)$. Since $\mc I_g(x,x)$ is finite, every orbit in $\mc I_g(x,x)$ must be finite, thus there exists an idempotent map $f : x \to x$ which is mapped to $g$. Since $\mc I$ is Cauchy complete, the splitting of $f$, which is a pair $x \gb y$, exists in $\mc I$. Since splitting of idempotency is an absolute colimit, it is preserved by $A$, hence we must have $Ay \cong a$.
\end{proof}

Thus, the (surjection, inclusion)-factorisation system on $\Topos_{\mr f}$ transports to the (essentially surjective upto retract, fully faithful)-factorisation system in $\Catf$. For any functor $A : \mc I \to \mc J$, it can be factored as $\mc I \to \hatt{\im A} \to \mc J$, where $\hatt{\im A}$ is the retract closure of images of $A$ in $\mc J$. Notice that to take retract closure is forced when working in $\Catf$, because the image $\im A$ might not be Cauchy complete.

The argument in the proof of Lemma~\Ref{lem:fullcloseretr} is typical for the class of finite Cauchy complete categories, and we will see further examples when we look at open functors to describe the duality for finite Heyting pretoposes.

However, the class of surjections is not entirely well-behaved in $\Catf$, because it is not stable under 2-pullbacks. For instance, consider the following category in $\Catf$
\[ \mc I :=
  \begin{tikzcd}
    \cha \ar[r, shift left, tail] & \cha \ar[l, shift left, two heads]
  \end{tikzcd}
\]
which can be viewed as the classifying category of retract pairs. Note that $\mc I$ has two points, which we denote as
\[ s,r : \mb 1 \to \mc I, \]
where $r$ denote the point on the right and $s$ the point on the left, which is a retract of $r$. By Lemma~\Ref{lem:surjessretra}, the point $r$ is actually \emph{surjective}, because the only other object is a retract of it. However, its 2-pullback along the point $s$ is \emph{empty},
\[
  \xymatrix{
    \emp \ar[r] \ar[d] & \mb 1 \ar[d]^r \\
    \mb 1 \ar[r]_s & \mc J
  }
\]
because the two points are \emph{not} isomorphic. The functor $\emp \to \mb 1$ is evidently \emph{not} a surjection. This motivates us to define the following stronger notion of surjection:

\begin{definition}
  We say a functor $A : \mc I \to \mc J$ is \emph{stably surjective}, if its pullback along any other functor is again surjective.
\end{definition}

It is easy to see that stably surjective morphisms are exactly \emph{essentially surjective functors}:

\begin{lemma}
  A functor $A : \mc I \to \mc J$ is stably surjective iff it is essentially surjective.
\end{lemma}
\begin{proof}
  The if direction is easy because essentially surjective functors are stable under 2-pullbacks. On the other hand, suppose $A$ is not essentially surjective, then we can find $a\in\mc J$ which does not lie in the essential image of $A$. This way, again we have a 2-pullback square
  \[
    \xymatrix{
      \emp \ar[r] \ar[d] & \mc I \ar[d]^A \\
      \mb 1 \ar[r]_a & \mc J
    }
  \]
  which implies $A$ cannot be stably surjective.
\end{proof}

This stronger notion will play an important role when we study open morphisms in the next section, and ultimately in constructing a 2-site to represent finitely presented Heyting pretoposes.

\begin{remark}
  Note that there are further factorisation systems in $\Topos_{\mr f}$, and all of them have corresopnding factorisation systems in $\Catf$. We record the following results without proof, where most of which can be seen as consequences of the results in~\cite{caramello2019denseness}:
  \begin{itemize}
  \item The (hyperconnected, localic)-factorisation system in $\Topos_{\mr f}$ corresponds to the (essentially surjective and full, faithful)-factorisation on $\Catf$. Note that Lemma~\Ref{lem:fullcloseretr} implies that this factorisation always produces Cauchy complete categories.
  \item The pair (terminally connected, \'etale) forms an orthogonal factorisation system in the 2-category $\mb{EssTopos}$ of Grothendieck toposes and essential geometric morphisms. An essential geometric morphism $f$ is terminally connected iff the further left adjoint $f_{!}$ of $f^{*}$ preserves the terminal object. Since all geometric morphisms between $\Fin$-bounded toposes are essential, they again form an orthogonal factorisation system in $\Topos_{\mr f}$, and this exactly corresopnds to the (initial, discrete opfibration)-factorisation system in $\Catf$.
  \end{itemize}
\end{remark}

\subsection{Open Functors and Duality of Finite Heyting Pretoposes}\label{subsec:openheyting}

As we have mentioned, it turns out that any finite pretopos is automatically Heyting. Let $\mb{HPretopos}_{\mr f}$ be the 2-category of \emph{finite Heyting pretoposes}, which is defined to be a wide subcategory of $\Pretopos_{\mr f}$ with Heyting functors between them. Recall that a geometric morphism is called \emph{open} if its inverse image functor is Heyting. According to Lemma~\Ref{lem:cohfinessgeom}, to find the dual of $\mb{HPretopos}_{\mr f}$, it is equivalent to characterise \emph{open} geometric morphisms between $\Fin$-bounded toposes.

In fact, Chapter C3.1 of~\cite{johnstone2002sketches} already contains a classification of open geometric morphisms between presheaf toposes which are induced by a functor between the underlying categories. This suffices for our purpose because any geometric morphism in $\Topos_{\mr f}$ arises this way. Concretely, a functor $A : \mc I \to \mc J$ in $\Catf$ induces an \emph{open} geometric morphism iff it satisfies the following combinatorial property: For any $x\in\mc I$ and any $f : Ax \to a$ in $\mc J$, there exists some $u : x \to y$ in $\mc I$ such that $a$ is a \emph{retract} of $Ay$ under $Ax$,
\[
  \begin{tikzcd}
    a && Ay \\
    & Ax
    \arrow["f", from=2-2, to=1-1]
    \arrow["Au"', from=2-2, to=1-3]
    \arrow[curve={height=6pt}, tail, from=1-1, to=1-3]
    \arrow[curve={height=6pt}, two heads, from=1-3, to=1-1]
  \end{tikzcd}
\]
From our characterisation of surjections, we can rephrase it as follows:

\begin{definition}
  We say a functor $A : \mc I \to \mc J$ is \emph{upwards surjective} if for any $x\in\mc I$, the induced functor
  \[ A_{x/} : \mc I_{x/} \to \mc J_{Ax/} \]
  is a surjection.
\end{definition}

Although open functors are not necessarily surjective, they will be once we restricts to any upwards part of the functor:

\begin{lemma}\label{lem:heresurjopen}
  $A : \mc I \to \mc J$ is open iff it is upwards surjective.
\end{lemma}
\begin{proof}
  Recall that by Proposition~\ref{lem:surjessretra}, a functor is a surjection iff it is essentially surjective upto retracts. It is then straight forward to see that openness is equivalent to being upwards surjective.
\end{proof}

This way, let us use $\Casf$ to denote the wide subcategory of $\Catf$ consisting of open, or equivalently upwards surjective, functors. As a consequence, we have the following duality result for finite Heyting pretoposes:

\begin{theorem}\label{thm:dualfinhp}
  The equivalence $\Catf\op \simeq \Pretopos_{\mr f}$ restricts to one 
  \[ \Casf\op \simeq \mb{HPretopos}_{\mr f} \]
  between the opposite 2-category of finite Cauchy complete categories with open functors and the 2-category of finite Heyting pretoposes.
\end{theorem}
\begin{proof}
  By Theorem~\ref{thm:finpredual}, $\mb{HPretopos}_{\mr f}$ is dual to the wide subcategory of $\Catf$ of those functors inducing open geometric morphisms. And as mentioned, the characterisation of open geometric morphisms as open functors is given in~\cite{johnstone2002sketches}.
\end{proof}

\begin{remark}
  Just like that the duality between finite Cauchy complete categories and finite pretoposes restricts to the localic case, so is our duality of open functors with Heyting morphisms. This means that we have a commuative diagramme as follows,
  \[
    \begin{tikzcd}
      \Por\op && \HAf \\
      & \Posf\op && {\DL_{\mr f}} \\
      \Casf\op && {\mb{HPretopos}_{\mr f}} \\
      & \Catf\op && {\Pretopos_{\mr f}}
      \arrow[tail, from=1-1, to=2-2]
      \arrow["\simeq", from=1-1, to=1-3]
      \arrow[tail, from=1-3, to=2-4]
      \arrow["\simeq"{pos=0.2}, from=2-2, to=2-4]
      \arrow[curve={height=-6pt}, hook, from=1-3, to=3-3]
      \arrow[tail, from=3-1, to=4-2]
      \arrow[curve={height=-6pt}, hook, from=1-1, to=3-1]
      \arrow[curve={height=-6pt}, from=3-1, to=1-1]
      \arrow["\dashv"{anchor=center, pos=0.35}, draw=none, from=3-1, to=1-1]
      \arrow[curve={height=-6pt}, from=3-3, to=1-3]
      \arrow["\dashv"{anchor=center, pos=0.35}, draw=none, from=3-3, to=1-3]
      \arrow[curve={height=-6pt}, hook, from=2-2, to=4-2]
      \arrow[curve={height=-6pt}, from=4-2, to=2-2]
      \arrow["\dashv"{anchor=center, pos=0.35}, draw=none, from=4-2, to=2-2]
      \arrow["\simeq", from=4-2, to=4-4]
      \arrow[curve={height=-6pt}, hook, from=2-4, to=4-4]
      \arrow["\simeq"{pos=0.2}, from=3-1, to=3-3]
      \arrow[tail, from=3-3, to=4-4]
      \arrow[curve={height=-6pt}, from=4-4, to=2-4]
      \arrow["\dashv"{anchor=center, pos=0.35}, draw=none, from=4-4, to=2-4]
    \end{tikzcd}
  \]
\end{remark}

In our finitary case, open functors have nicer properties. As mentioned in Section~\Ref{subsec:surjection}, surjections are not necessarily stable under 2-pullbacks in $\Catf$. One important observation for open functors in $\Catf$ is that, open morphisms $A : \mc I \to \mc J$ in $\Catf$ are not just upwards surjective, but \emph{upwards stably surjective}:

\begin{proposition}\label{prop:upwardesssurj}
  For any functor $A : \mc I \to \mc J$ in $\Catf$, $A$ is open iff it is upwards stably surjective, i.e. for any $x\in\mc I$, $A_{x/} : \mc I_{x/} \to \mc J_{Ax/}$ is essentially surjective.
\end{proposition}
\begin{proof}
  The if direction is trivial. Suppose $A : \mc I \to \mc J$ is open, consider any $f : Ax \to a$. Now by assumption, there exists some $u_1 : x \to x_1$ and a retract pair $s_1 : a \gb Ax_1 : r_1$, inducing the following diagramme,
  \[
    \begin{tikzcd}
      a \\
      Ax & {Ax_1}
      \arrow["f", from=2-1, to=1-1]
      \arrow["{Au_1}"', from=2-1, to=2-2]
      \arrow[shift right, tail, from=1-1, to=2-2]
      \arrow[shift right, two heads, from=2-2, to=1-1]
    \end{tikzcd}
  \]
  Consider the following recursive process:
  \begin{itemize}
  \item If the retract pair $s_1 : a \gb Ax_1 : r_1$ is trivial, viz. an isomorphism, then we are done, since now $f$ is isomorphic to $Au_1$ in $\mc J_{Ax/}$. \vspace{.5ex}
  \item If the retract pair is not trivial, then we may apply the openness condition on $r_1 : Ax_1 \surj a$ again, to obtain another retract pair $s_2 : a \gb Ax_2 : r_2$,
    \[
      \begin{tikzcd}
        a \\
        {Ax_1} & {Ax_2}
        \arrow["r_1", two heads, from=2-1, to=1-1]
        \arrow["{Au_2}"', from=2-1, to=2-2]
        \arrow[shift right, tail, from=1-1, to=2-2]
        \arrow[shift right, two heads, from=2-2, to=1-1]
      \end{tikzcd}
    \]
    Note that $Au_2$ cannot be an isomorphism, because if it is, then by the fact that $s_2r_1 = Au_2$, $r_1$ must be monic; but $r_1$ is already split epic, thus it is monic iff it is an isomoprhism, which by our assumption it isn't. Thus, $Au_2$ is not an isomorphism. Now we may run the algorithm again recursively on $r_2$.
  \end{itemize}
  Notice that if our algorithm generates the following sequence
  \[
    \xymatrix@1{
      Ax_1 \ar[r]^{Au_2} & Ax_2 \ar[r]^{Au_3} & \cdots \ar[r]^{Au_n} & Ax_n \ar[r]^{Au_{n+1}} & \cdots
    }
  \]
  with their corresponding non-trivial retract pairs $s_i : a \gb Ax_i : r_i$ for $i \le n$, none of the composite $Au_{i+j} \circ \cdots \circ Au_{i+1} \circ Au_i$ can be an isomorphism, again due to the fact that $r_i$ cannot be an isomorphism by definition. Thus, either the algorithm stops at some point, which means we have obtained a trivial retract pair, $s_n : a \cong Ax_n : r_n$ for some $n$, and it follows that
  \[ f \cong Au_n \circ \cdots \circ Au_1. \]
  Or the sequence is ultimately periodic, viz. there is a cofinal part of the sequence which can be viewed as constant on some $Av : Ay \to Ay$ generating the same retract pair $s : a \gb Ay : r$, as indicated as follows,
  \[
    \begin{tikzcd}
      a & a \\
      Ay & Ay
      \arrow["s"', shift right, tail, from=1-1, to=2-1]
      \arrow["r"', shift right, two heads, from=2-1, to=1-1]
      \arrow["Av"', from=2-1, to=2-2]
      \arrow[equal, from=1-1, to=1-2]
      \arrow["s"', shift right, tail, from=1-2, to=2-2]
      \arrow["r"', shift right, two heads, from=2-2, to=1-2]
    \end{tikzcd}
  \]
  Now by the same argument in the proof of Lemma~\Ref{lem:fullcloseretr}, we may assume $v$ to be idempotent and split $y$ in $\mc I$ as it is Cauchy complete. Since any functor preserves splitting of idempotencies, it follows that the retract pair $r : Ay \gb a : s$ also lies in the essential image of $A$.
\end{proof}

This result allows us to see more directly why open functors are stable under pullback in $\Catf$ in the next section, which will be crucial for our construction of a $(2,1)$-site on $\Casf\iso$.

\section{Stack Representation}\label{sec:2-topos}

In the previous section we have shown that there are dualty results as follows,
\[
  \begin{tikzcd}
    \mb{HPretopos}_{\mr f}\op \ar[d, tail] \ar[r, "\simeq"] & \Casf \ar[d, tail] \\
    \Pretopos_{\mr f}\op \ar[r, "\simeq"'] & \Catf
  \end{tikzcd}
\]
which is a duality for the full 2-categorical version. By restricting the above 2-categories to their corresopnding $(2,1)$-categories, similarly we have
\[
  \begin{tikzcd}
    \HP\op \ar[d, tail] \ar[r, "\simeq"] & \Casf\iso \ar[d, tail] \\
    \Pt\op \ar[r, "\simeq"'] & \Catf\iso
  \end{tikzcd}
\]

As mentioned in Section~\Ref{sec:intro}, our task in this section is to first put a suitable Grothendieck topology on $\Casf\iso$, the $(2,1)$-category of finite Cauchy complete categories with open functors, so as to realise it as a $(2,1)$-site. Notice that as a $(2,1)$-category, $\Casf\iso$ may not admit all finite 2-limits, but $\Catf\iso$ does. Even if some limit does exist in $\Casf\iso$, the inclusion $\Casf\iso\inj\Catf\iso$ not necessarily preserves it. In the following texts, whenever we are taking finite 2-limits of finite Cauchy complete categories, we are always taking the 2-limits in $\Catf\iso$, even though the 2-limit might exists in $\Casf\iso$.

However, 2-colimits in $\Catf\iso$ \emph{do} coincide with that in $\Casf\iso$, because they are dual to $\Pt$ and $\HP$ respectively, and the two $(2,1)$-categories compute the same finite 2-limits. This is due to the fact that $\mb{HPretopos}$ is \emph{monadic} over $\Pretopos$.

By following a logical intuition, we will be thinking about $\Casf\iso$ as the $(2,1)$-category of \emph{finite Kripke frames}. Our task in Section~\Ref{subsec:sitekripke} is to construct a Grothendieck topology $J$ on $\Casf\iso$, that in some sense fully reflects its underlying logical information. In Section~\Ref{subsec:stackrepresent} we will first construct the stack representation functor
\[ \Phi_- : \HPfp\op \to \St(\Casf\iso,J), \]
where it intuitively sends each finitely presented Heyting pretopos to the stack of its \emph{Kripke models}.

We think of the above functor as \emph{gros representation}, which embeds the opposite of the total $(2,1)$-category of finitely presented Heyting pretoposes in $\St(\Casf\iso,J)$. Any such representation also induces a \emph{petit representation}, which embeds a particular Heyting pretopos $\mc H$ in $\St(\Casf\iso,J)$ via taking slices. Concretely, this is obtained via first embed $\mc H$ as
\[ \mc H \to \mc H/\HPfp\op, \]
by sending any $\varphi\in\mc H$ to the pullback onto the slice over it $\mc H \to \mc H/\varphi$, and then compose with the gros representation to get a functor
\[ \mc H \to \St(\Casf\iso,J)/\Phi_{\mc H}. \]
We will show that this functor actually lands in \emph{discrete opfibrations} over $\Phi_{\mc H}$, and will be a \emph{faithful Heyting functor} from $\mc H$ to the discrete opfibration slice over $\Phi_{\mc H}$. These will be developed in Section~\Ref{subsec:grospetit}.

\subsection{The $(2,1)$-Site of Finite Kripke Frames}\label{subsec:sitekripke}

For a general discussion on 2-categorical notions of sites and stacks, we refer the readers to~\cite{street1982twosheaf}. The Grothendieck topology $J$ we put on $\Casf\iso$ is obtained by choosing a specific family of covers. For us, the relevant notion lies in \emph{open surjections}. For any $\mc P\in\Casf\iso$ and any family of maps $\set{A_i : \mc Q_i \to \mc P}$ in $\Casf\iso$, we say it is a \emph{$J$-cover} if it contains a \emph{finite} family which is \emph{jointly surjective}. Since we work in the $(2,1)$-category $\Casf\iso$, this is equivalent to the existence of a finite set $I$ such that the induced coproduct
\[ [A_i]_{i\in I} : \coprod_{i\in I}\mc Q_i \surj \mc P \]
is an open surjection. Notice that it is necessarily open, thus the condition solely lies in surjectivity. 

To see this gives a well-defined Grothendieck coverage $J$ on $\Casf\iso$, we first establish certain stability result. Our first observation is that, due to our characterisation of open maps in Proposition~\ref{prop:upwardesssurj}, open maps are indeed stable under 2-pullback in $\Catf\iso$:

\begin{lemma}\label{lem:openstab}
  In $\Catf\iso$, if $A : \mc Q \to \mc P$ is open, then its pullback along any $B : \mc R \to \mc P$ will again be open,
  \[
    \xymatrix{
      \mc Q \times_{\mc P} \mc R \ar@{}[dr]|{\cong} \ar[r]^D \ar[d]_C & \mc Q \ar[d]^A \\
      \mc R \ar[r]_B & \mc P
    }
  \]
\end{lemma}
\begin{proof}
  Recall by Proposition~\ref{prop:upwardesssurj}, a morphism is open iff it is upwards stably surjective, thus the results will follow from a computation of slices. For any $(q,r,\alpha)\in\mc Q \times_{\mc P} \mc R$ with $\alpha : Aq \cong Br$, let $p$ be some element in $\mc P$ isomorphic to both $Aq$ and $Br$. This way, we would also have a 2-pullback square on slices,
  \[
    \xymatrix{
      (\mc Q \times_{\mc P} \mc R)_{(q,r)/} \ar@{}[dr]|{\cong} \ar[r]^D \ar@{->>}[d]_C & \mc Q_{q/} \ar@{->>}[d]^A \\
      \mc R_{r/} \ar[r]_B & \mc P_{p/}
    }
  \]
  By $A$ being open, we know that $\mc Q_{q/} \surj \mc P_{p/}$ is stably surjective, thus so is its pullback. This implies $C$ is open as well.
\end{proof}

\begin{remark}
  In topos theory, it is true more generally that open geometric morphisms are stable under pullback along bounded geometric morphisms, e.g. see Chapter C3.1 in~\cite{johnstone2002sketches}. In the finitary case of our interest, as shown in Lemma~\Ref{lem:openstab} this stability result can be seen as a direct consequence of the combinatorial properties of open maps established in Proposition~\ref{prop:upwardesssurj}. In contrast, for a general functor $F : \mc C \to \mc D$ which is open, its pullback along any other functor $G : \mc B \to \mc D$ computed in $\Cat$ is by no means necessarily still open.
\end{remark}

As an easy consequence, so is open surjections:

\begin{lemma}\label{lem:opensurjstab}
  A functor in $\Catf\iso$ is an open surjection iff it is stably surjective and upwards stably surjective, thus in particular open surjections are also stable under 2-pullbacks.
\end{lemma}
\begin{proof}
  We only need to show any open surjection is stably surjective, viz. essentially surjective. Suppose $A : \mc I \to \mc J$ is an open surjection, then for any $a\in\mc J$, by $A$ being surjective we can find a retract pair $s : a \gb Ax : r$. Now by $A$ being open, the functor
  \[ A_{x/} : \mc I_{x/} \surj \mc J_{Ax/} \]
  is essentially surjective, thus we can find $u : x \to y$ in $\mc I_{x/}$ such that $Au \cong r$ in $\mc J_{Ax/}$. In particular, this implies $a \cong Ay$, thus $A$ is essentially surjective.
\end{proof}

With such a stability result, we can now easily see that $J$ is a Grothendieck coverage on the $(2,1)$-category $\Casf\iso$:

\begin{lemma}
  $J$ is a Grothendieck coverage on the $(2,1)$-category $\Casf\iso$.
\end{lemma}
\begin{proof}
  Any equivalence is evidently an open surjection. For any $J$-covering family on $\mc P$, suppose it contains a jointly surjective family $\set{A_i : \mc Q_i \to \mc P}_{i\in I}$. For any $B : \mc R \to \mc P$, the family of pullbacks $\set{B^{*}A_i : \mc Q_i \times_{\mc P}\mc R \to \mc R}_{i\in I}$ by Lemma~\ref{lem:opensurjstab} is again jointly surjective, because coproducts in the $(2,1)$-category $\Catf$ are disjoint. Thus, $J$ satisfies stability. Transitivity is evident from the fact that open surjections are closed under composition.
\end{proof}

This way, we have constructed a $(2,1)$-site $(\Casf\iso,J)$. Before studying the $(2,1)$-topos generated by this $(2,1)$-site, we observe that in practice it is often easier to work with a \emph{dense subsite} of it. Let $\Casfp\iso$ be the full sub 2-category of $\Casf\iso$ consisting of those finite Cauchy complete categories which has an \emph{initial object}. If $\mc P \in \Casfp\iso$ has an initial object, we will also call it \emph{rooted}, and use $*$ to denote the root, or the initial object, of $P$.

We show that objects in $\Casfp\iso$ are \emph{regular} for the site $(\Casf\iso,J)$:

\begin{lemma}\label{lem:rootedsurj}
  If $\mc P\in\Casfp\iso$ is rooted, then $A : \mc Q \to \mc P$ in $\Casf\iso$ is an open surjection iff $*\in\mc P$ lies in the essential image of $A$.
\end{lemma}
\begin{proof}
  By Lemma~\ref{lem:opensurjstab}, an open surjection $A : \mc Q \to \mc P$ must be essentially surjective, thus $*$ does lie in the essential image of $A$. On the other hand, if there exists $q\in\mc Q$ with $Aq \cong *$, then since $A$ is open, we have a stable surjection
  \[ A_{q/} : \mc Q_{q/} \surj \mc P_{*/} \cong \mc P \]
  This particularly implies that $A : \mc Q \surj \mc P$ is also essentially surjective, thus $A$ is an open surjection.
\end{proof}

\begin{corollary}
  If $\mc P\in\Casfp\iso$ is rooted, then any covering family on $\mc P$ contains a single open surjection.
\end{corollary}
\begin{proof}
  Consider any finite family $\set{A_i : \mc Q_i \to \mc P}$ which is jointly surjective,
  \[ [A_i] : \coprod\mc Q_i \surj \mc P. \]
  Now since $*\in\mc P$ lies in the essential image, there must exist some $i$ that $*\in\mc P$ lies in the essential image of $A_i : \mc Q_i \to \mc P$. By Lemma~\ref{lem:rootedsurj}, $A_i$ is already an open surjection.
\end{proof}

Furthermore, any object in $\Casf\iso$ can be covered by these regular objects:

\begin{lemma}\label{lem:rooteddens}
  Any $\mc I\in\Casf\iso$ admits a $J$-covering family $\set{\mc P_i \to \mc I}_{i\in I}$ where each $\mc P_i$ is rooted.
\end{lemma}
\begin{proof}
  For any $x\in\mc I$, consider the functor $\mc I_{x/} \to \mc I$. All such maps are open, and the finite family $\set{\mc I_{x/} \to \mc I}_{x\in\mc I}$ evidently is a $J$-covering family. 
\end{proof}

\begin{corollary}
  The inclusion $(\Casfp\iso,J) \hook (\Casf\iso,J)$ induces an equivalence on the $(2,1)$-toposes of stacks over them,
  \[ \St(\Casfp\iso,J) \simeq \St(\Casf\iso,J). \]
\end{corollary}
\begin{proof}
  Lemma~\Ref{lem:rooteddens} implies $\Casfp$ is a $J$-dense full subcategory of $\Casf$, and the desired result follows from a comparison lemma for 2-toposes, e.g. see Theorem 3.8 in~\cite{street1982twosheaf}.
\end{proof}

Thus, we can always work with the site $(\Casfp\iso,J)$, where now for any rooted $\mc P$, a family is $J$-covering iff it contains a single open surjection. This means that $J$ is a \emph{regular} topology on $\Casfp$, which helps us greatly to compute 2-colimits of stacks.

\subsection{Stack Representation of Heyting Pretoposes}\label{subsec:stackrepresent}

Now for a general Heyting pretopos $\mc H$, we know that they are not complete w.r.t. classical models. For our focus of finitely presented Heyting pretopos, we will consider its models in all \emph{finite Kripke frames}, i.e. models valued in the Heyting pretoposes of the form $\Fin[\mc P]$, for $\mc P\in\Casfp\iso$. In particular, \emph{any} Heyting pretopos $\mc H$ induces a prestack $\Phi_{\mc H}$ on the $(2,1)$-category $\Casfp\iso$ as follows:
\[ \Phi_{\mc H} := \HPt(\mc H,\Fin[-]), \]
where $\HPt$ is the $(2,1)$-category of Heyting pretoposes. In particular, $\Phi_{\mc H}$ is groupoid valued. The action of $\Phi_{\mc H}$ on $A : \mc Q \to \mc P$ in $\Casfp\iso$ is given by post-composition with the inverse image $A^{*} : \Fin[\mc P] \to \Fin[\mc Q]$. For any finite Kripke frame $\mc P\in\Casfp\iso$, we think of $\Phi_{\mc H}(\mc P)$ as the \emph{groupoid of models} of $\mc H$ on the $\mc P$, or simply \emph{$\mc P$-models} of $\mc H$. In particular, the pre-stack records the information of all \emph{finite} models of $\mc H$ on \emph{finite} Kripke frames.

Our first task then is to show that $\Phi_{\mc H}$ will indeed be a \emph{stack} for the topology $J$ we have chosen on $\Casfp\iso$, and this in fact can be seen as a consequence of the descent theorem for toposes. 

\begin{lemma}\label{lem:modelstack}
  For any Heyting pretopos $\mc H$, $\Phi_{\mc H}$ is a stack of groupoids on the $(2,1)$-site $(\Casfp\iso,J)$.
\end{lemma}
\begin{proof}
  Recall the descent theorem for toposes first given in~\cite{joyal1984extension}, which states that open surjections are of effective descent for toposes. Since the descent theorem can be proven in purely elementary means, e.g. see~\cite{moerdijk1985descent}, in our finitary case this implies that the following diagramme will be a 2-colimit in $\Topos_{\mr f}$ under the equivalence with $\Catf$,
  \[
    \begin{tikzcd}
      \mc Q \times_{\mc P} \mc Q \times_{\mc P} \mc Q \ar[r, shift left = 2] \ar[r, shift right = 2] \ar[r] & \mc Q \times_{\mc P} \mc Q \ar[r, shift right] \ar[r, shift left] & \mc Q \ar[r, two heads, "A"] & \mc P
    \end{tikzcd}
  \]
  By duality, this becomes a 2-limit in $\Pretopos_{\mr f}$. In fact, the dual will be a 2-limit in $\Pretopos$, since the full inclusion $\Pretopos_{\mr f} \hook \Pretopos$ evidently preserves and creates finite 2-limits. As we've mentioned, 2-limits in $\mb{HPretopos}$ are computed the same as in $\Pretopos$, thus the dual diagramme is also a 2-limit for Heyting pretoposes.
\end{proof}

\begin{remark}
  In fact for our finitary case, it is fairely straight forward to directly prove that the above diagramme is a 2-colimit in $\Catf$ when $A$ is an open surjection, which avoids the more technical topos-theoretic proofs. We leave the details for the interested readers.
\end{remark}

This way, we have realised any Heyting pretopos $\mc H$ as a stack of groupoid of finite models over the $(2,1)$-category of finite Kripke frames. However, due to our restriction on finite models on finite Kripke frames, we cannot expect that this stack $\Phi_{\mc H}$ we associate to $\mc H$ will reveal all the information for an arbitrary Heyting pretopos. This way, we introduce the relevant class of \emph{finitely presentable} Heyting pretoposes:

\begin{definition}
  We say a Heyting pretopos $\mc H$ is \emph{finitely presentable}, if it is complete w.r.t. finite models on finite Kripke frames, i.e. the following canonical functor is \emph{faithful}
  \[ \mc H \inj \prod_{\mc P\in\Casfp,F\in\Phi_{\mc H}(\mc P)}\Fin[\mc P], \]
\end{definition}

Let us use $\HPfp$ to denote the $(2,1)$-category of finitely presentable Heyting pretoposes. From the above definition, $\HPfp$ basically contains those sub Heyting pretoposes of products of finite ones. We thus view our stack representation as a 2-functor as follows,
\[ \Phi_- : \HPfp\op \to \St(\Casfp\iso,J), \]
This way, we have indeed embedded our interested class of Heyting pretoposes into the 2-category of stacks on the 2-site $(\Casfp\iso,J)$, contravariantly. 

Evidently, finite Heyting pretoposes are finitely presentable. More generally, the \emph{free} Heyting pretopos generated by a finite pretopos will also be finitely presentable. For any finite Pretopos $\Fin[\mc I]$, the free Heyting pretopos generated by it will be denoted as $\Hey[\mc I]$. In particular, for any Heyting pretopos $\mc H$, we have a natural equivalence of categories
\[ \Pretopos(\Fin[\mc I],\mc H) \simeq \mb{HPretopos}(\Hey[\mc I],\mc H). \]
where the equivalence is induced by an inclusion $\Fin[\mc I] \inj \Hey[\mc I]$. For instance, the \emph{initial} Heyting pretopos can be identified as the free Heyting pretopos $\Hey \simeq \Hey[\mb 1]$, which is freely generated by $\Fin = \Fin[\mb 1]$.

It is instructive to look at the stack representation of these free Heyting pretoposes. Notice that for any $\mc P\in\Casfp\iso$ we have the following equivalences,
\begin{align*}
  \Phi_{\Hey[\mc I]}(\mc P)
  &\simeq \HPfp(\Hey[\mc I],\Fin[\mc P]) \\
  &\simeq \Ptfp(\Fin[\mc I],\Fin[\mc P]) \\
  &\simeq \Catf\iso(\mc P,\mc I).
\end{align*}
We also write the stack $\Catf\iso(-,\mc I)$ on $\Casfp\iso$ as $\uv{\mc I}$. The above equivalence shows that for any finite Cauchy complete category $\mc I$, we have
\[ \Phi_{\Hey[\mc I]} \simeq \uv{\mc I} \]
in the $(2,1)$-category $\St(\Casfp\iso,J)$, and thus by Lemma~\ref{lem:modelstack}, all prestacks of the form $\uv{\mc I}$ are in fact stacks for the topology $J$.

\begin{remark}
  In fact, one can show that $J$ is the \emph{canonical} topology on $\Casfp\iso$, i.e. the largest topology such that all representable prestacks are stacks. One can write down a proof of this fact very similar to the posetal case presented in~\cite{ghilardi2013sheaves}. But since this is irrelevant for our investigation, we omit it here.
\end{remark}

We also note that the 2-category of finite Heyting pretoposes is closed under slices, thus as a consequence it is closed under finite quotients of finitely presentable Heyting pretoposes:

\begin{lemma}
  If $\mc H$ is a finite Heyting pretopos, then so is the slice $\mc H/\varphi$ for any $\varphi\in\mc H$.
\end{lemma}
\begin{proof}
  Recall that the slice $\mc H/\varphi$ classifies $\varphi$-points of models of $\mc H$,
  \[ \HP(\mc H/\varphi,\Fin[\mc P]) \cong \sum_{F\in\Phi_{\mc H}(\mc P)}\Gamma F\varphi \cong \sum_{F\in\Phi_{\mc H}(\mc P)}(F\varphi)_{*}. \]
  Here $\Gamma$ is the global section functor, and the sum notation $\sum_{F\in\Phi_{\mc H}(\mc P)}\Gamma F\varphi$ denote the category of pairs $(F,\alpha)$ with $F \in \Phi_{\mc H}(\mc P)$ a model of $\mc H$ in $\Fin[\mc P]$, while $\alpha : 1 \to F\varphi$ is a global point of $F\varphi$ in $\Fin[\mc P]$. This is well-known, and one explicit proof can be found e.g. in Chapter 2.3 of~\cite{breiner2013duality}.\footnote{The proof there is given for pretoposes, but the same proof works for Heyting pretoposes as well.} The final isomorphism is due to the fact that $\mc P$ has an initial object, thus taking the global section is equivalent of evaluating on the root.

  Now if we have a proper subterminal object $U$ in $\mc H/\varphi$, it corresponds to some proper subobject $U \inj \varphi$. By $\mc H$ being finite, we can find a finite model $F : \mc H \to \Fin[\mc P]$ such that $FU \subsetneq F\varphi$ is proper. Then we may choose some $a\in (F\varphi)_{*}$ such that $a\not\in (FU)_{*}$. The pair $(F,a)$ then gives us a model of $\mc H/\varphi$ which falsifies $U$.
\end{proof}

\begin{remark}
  It might be tempting to say that \emph{all} finitely presentable Heyting pretoposes are finite quotients of some free one of the form $\Hey[\mc I]$. However we do not know whether this is true or not at this point.
\end{remark}

\subsection{From Gros to Petit}\label{subsec:grospetit}

Following the \emph{``gros-petit''} philosophy in topos theory, given any finitely presented Heyting pretopos $\mc H$, the \emph{petit} category $\mc H$ can be embedded into the gros category $\mc H/\HPfp\op$ of the opposite of the 2-category of finite Heyting pretoposes under $\mc H$,
\[ \mc H \inj \mc H/\HPfp\op, \]
sending any $\varphi\in\mc H$ to the pullback onto the slice over it,
\[ \varphi^{*} : \mc H \to \mc H/\varphi. \]
By composing with the stack representation $\Phi_- : \HPfp\op \to \St(\Casfp\iso,J)$ functor, which we may as well denote it as the \emph{gros representation functor}, we naturally get a \emph{petit representation} of $\mc H$ as follows,
\[ \mc H \to \St(\Casfp\iso,J)/\Phi_{\mc H}. \]
On the left we have a \emph{1-category} $\mc H$, thus we expect that this representation to land in \emph{discrete} objects over $\Phi_{\mc H}$ internally in $\St(\Casfp\iso,J)$.

Following the general development of fibrations in a 2-category in~\cite{street1974fibration}, in a 2-category $\mc K$ a morphism $f : A \to B$ is called a \emph{discrete opfibration} if for any $C\in\mc K$, the induced functor $\mc K(C,A) \to \mc K(C,B)$ is a discrete opfibration. Notice that for a fibration, we would also require that for any $g : D \to C$ the base change square will be a morphism of fibrations. However, since for discrete opfibrations every morphism in the over category will be Cartesian, this condition is trivially satisfied. Equivalently, $f : A \to B$ is a discrete fibration iff it is a discrete object in the 2-category of opfibrations over $B$, and discrete opfibrations over some object naturally forms a \emph{1-category}.

In the $(2,1)$-category of stacks over $(\Casfp\iso,J)$, to check opdiscrete fibrations it suffices to check on representables. For any stack $\mc F$ in $\St(\Casfp\iso,J)$, we use $\opDisc(\mc F)$ to denote the 1-category of opdiscrete fibrations on $\mc F$. Furthermore, the category $\opDisc(\mc F)$ will in fact be a \emph{Heyting category}, e.g. see~\cite{nlab:heyting_2-category}.

We first show the petit representation functor indeed lands in discrete opfibrations in $\St(\Casfp\iso,J)$:

\begin{lemma}\label{lem:petitdisc}
  The petit representation restricts to a functor as follows:
  \[ \mc H \to \mb{opDisc}(\Phi_{\mc H}). \]
\end{lemma}
\begin{proof}
  Given any $\varphi\in\mc H$, by construction for any $\mc P\in\Casfp\iso$, the morphism
  \[ \Phi_{\mc H/\varphi}(\mc P) \cong \sum_{F\in\Phi_{\mc H}(\mc P)}\Gamma F\varphi \cong \sum_{F\in\Phi_{\mc H}(\mc P)}(F\varphi)_{*} \to \Phi_{\mc H}(\mc P) \]
  is given by the projection map. Evidently, for any $F\in\Phi_{\mc H}(\mc P)$, the fibre of $F$ over this projection is the \emph{discrete set} $(F\varphi)_{*}$. Furthermore, for any $\eta : F \cong G$, this naturally provides a map $(F\varphi)_{*} \to (G\varphi)_{*}$. This shows for any $\mc P$ the above map morphism is a discrete opfibration, thus $\Phi_{\mc H/\varphi} \to \Phi_{\mc H}$ is a discrete opfibration in $\St(\Casfp\iso,J)$.
\end{proof}

More importantly, when restricting the codomain the discrete opfibrations over $\Phi_{\mc H}$, the petit representation will be a \emph{faithful Heyting functor} for any finitely presented Heyting pretopos:

\begin{theorem}\label{thm:petitheyting}
  For any $\mc H$ in $\HPfp$, the petit representation
  \[ \mc H \inj \mb{opDisc}(\Phi_{\mc H}) \]
  is a faithful Heyting functor.
\end{theorem}
\begin{proof}
  Faithfulness is basically built in in our definition of finitely presented Heyting pretoposes. Once we know it is a Heyting functor, it suffices to show for any proper subterminal object $\varphi$, it is not mapped to the identity on $\Phi_{\mc H}$. However, by $\mc H$ being finitely presented, there exists some $\mc P\in\Casfp\iso$ and a Heyting functor $F : \mc H \to \Fin[\mc P]$ that $F\varphi \inj 1$ is a proper subterminal object in $\Fin[\mc P]$. This implies $(F\varphi)_{*}$ must be empty, thus $F\not\in\Phi_{\mc H/\varphi}(\mc P)$. Hence, the induced injection $\Phi_{\mc H/\varphi} \inj \Phi_{\mc H}$ must also be proper. Hence, the remaining is to show it is a Heyting functor.
  
  We first show it preserves finite limits. Preserving the terminal objects is evident, since we have
  \[ \mc H/1 \simeq \mc H \mapsto \Phi_{\mc H} = \Phi_{\mc H}. \]
  For any $\alpha : \psi \to \varphi$ and $\beta : \chi \to \varphi$ in $\mc H$, we have the following equivalences,
  \begin{align*}
    &\, (\Phi_{\mc H/\psi} \times_{\Phi_{\mc H/\varphi}} \Phi_{\mc H/\chi})(\mc P) \\
    \cong &\, \scomp{(F,a) \in \sum_{F\in\Phi_{\mc H}(\mc P)}(F\psi)_{*},(G,b) \in \sum_{G\in\Phi_{\mc H}(\mc P)}(G\chi)_{*}}{F = G \conjt F\alpha a = G\beta b} \\
    \cong & \sum_{H\in\Phi_{\mc H}(\mc P)}(H\psi)_{*} \times_{(H\varphi)_{*}} (H\chi)_{*} \\
    \cong & \sum_{H\in\Phi_{\mc H}(\mc P)}H(\psi \times_{\varphi}\chi)_{*} \\
    \cong &\, \Phi_{\mc H/\psi \times_{\varphi} \chi}(\mc P)
  \end{align*}
  The first isomorphism is due to the fact that pullbacks in $\opDisc(\Phi_{\mc H})$ are computed point-wise; the second and the last isomorphisms have used the computation of models of slice Heyting pretoposes mentioned in the proof of Lemma~\ref{lem:petitdisc}; the third holds due to the fact that both $H$ and evaluation on the root preserve pullbacks.

  Next, we show that for any $X\in\mc H$, the induced map on subobjects
  \[ \sub_{\mc H}(X) \to \sub_{\mb{opDisc}(\Phi_{\mc H})}(\Phi_{\mc H/X}) \]
  is a distributive lattice morphism. Now subobjects of $\Phi_{\mc H/X}$ in $\opDisc(\Phi_{\mc H})$ coincides with subobjects of $\Phi_{\mc H/X}$ which are discrete opfibrations, e.g. see~\cite{nlab:heyting_2-category}, thus we have the following isomorphism
  \[ \sub_{\mb{opDisc}(\Phi_{\mc H})}(\Phi_{\mc H/X}) \cong \sub_{\mb{opDisc}(\Phi_{\mc H/X})}(\Phi_{\mc H/X}). \]
  This means it suffices to consider when $X$ is the terminal object. Given subterminal objects $\varphi,\psi\inj 1$ in $\mc H$, by definition for any $\mc P\in\Casfp\iso$,
  \begin{align*}
    \Phi_{\mc H/\varphi\wedge\psi}(\mc P)
    &\cong \scomp{F\in\Phi_{\mc H}(\mc P)}{F(\varphi \wedge \psi)_{*} \cong 1} \\
    &\cong \scomp{F\in\Phi_{\mc H}(\mc P)}{(F\varphi)_{*}} \cap \scomp{F\in\Phi_{\mc H}(\mc P)}{(F \models \psi)_{*}} \\
    &\cong \Phi_{\mc H/\varphi} \cap \Phi_{\mc H/\psi}
  \end{align*}
  Again, the second isomorphism holds since any such $F$ preserves finite limits. This shows the petit representation preserves meets in subobject lattices. Completely similarly it also preserves disjunctions. However, this further relies on the fact that $J$ is a \emph{regular} topology on $\Casfp\iso$, thus disjunction in subobject lattices of $\opDisc(\Phi_{\mc H})$ are also computed point-wise.

  To show it is a map between Heyting categories, we also need to show that the functor commutes with exisential and universal quantifiers. Given any $\alpha : Y \to X$ and $\varphi \inj Y$, for the case of existential quantification,
  \begin{align*}
    &\,\Phi_{\mc H/\exists_{\alpha}\varphi}(\mc P) \\
    \cong &\,\scomp{(F,a) \in \sum_{F \in \Phi_{\mc H}(\mc P)}(FX)_{*}}{a \in (F\exists_{\alpha}\varphi)_{*}} \\
    \cong &\,\scomp{(F,a) \in \sum_{F \in \Phi_{\mc H}(\mc P)}(FX)_{*}}{a \in \exists_{(F\alpha)_{*}}(F\varphi)_{*}} \\
    \cong &\,\scomp{(F,a) \in \sum_{F \in \Phi_{\mc H}(\mc P)}(FX)_{*}}{\exists b\in (F\varphi)_{*}.\alpha b = a} \\
    \cong &\,\exists_{\Phi_{\mc H/\alpha}}\Phi_{\mc H/\varphi}(\mc P)
  \end{align*}
  The second isomorphism holds since $F$ preserves existential quantifier, and that again due to the topology $J$ being regular, existential quantification is also computed point-wise thus preserved by $\ev_{*}$; the last isomorphism holds for the same reason.

  Finally, for universal quantification, we have
  \begin{align*}
    &\Phi_{\mc H/\forall_{\alpha}\varphi}(\mc P) \cong \sum_{F\in\Phi_{\mc H}(\mc P)}(\forall_{F\alpha}F\varphi)_{*} \\
    \cong &\sum_{F\in\Phi_{\mc H}(\mc P)}\scomp{x\in(FX)_{*}}{\forall p\in\mc P.\forall y\in(FY)_p.\alpha y = px \nt y\in(F\varphi)_p}
  \end{align*}
  The last isomorphism is basically by Kripke-Joyal semantics, for a reference see Chapter VI of~\cite{maclane2012sheaves}. On the other hand, if we compute the universal quantification in $\mb{opDisc}(\Phi_{\mc H})$, again by Kripke-Joyal semantics we have
  \begin{alignat*}{2}
    &\,\forall_{\Phi_{\mc H/\alpha}}\Phi_{\mc H/\varphi}(\mc P) \\
    \cong &\,\{\,(F,x)\in\Phi_{\mc H/X}(\mc P)\,\mid\,&&\ub{A}{\mc Q \to \mc P}\forall(G,y) \in \Phi_{\mc H/Y}(\mc Q)\\
    & \, && G = A^{*}F \conjt \alpha y = A^{*}x \nt (G,y) \in \Phi_{\mc H/\varphi}(\mc P)\} \\
    \cong &\,\{\,(F,x)\in\Phi_{\mc H/X}(\mc P)\,\mid\,&&\ub{A}{\mc Q \to \mc P}\forall y\in (A^{*}FY)_{*}\\
    & \, && \alpha y = A^{*}x \nt y \in (A^{*}F\varphi)_{*}\}
  \end{alignat*}
  However, for any such $A : \mc Q \to \mc P$, we know that
  \[ y \in (A^{*}F\varphi)_{*} \eff y \in (F\varphi)_p. \]
  where $p = A*$. Hence, the above two subobjects are isomorphic, which shows that the petit representation also commutes with the universal quantification. This completes the proof that the petit representation is a faithful Heyting functor.
\end{proof}

Thus, according to Theorem~\ref{thm:petitheyting}, any finitely presented Heyting pretopos $\mc H$ now can be viewed as a sub Heyting category of the discrete opfibred objects over $\Phi_{\mc H}$ in $\St(\Casfp\iso,J)$. As we've mentioned, when $\mc H$ is free of the form $\Hey[\mc I]$ for some $\mc I\in\Catf\iso$, $\Phi_{\Hey[\mc I]}$ is isomorphic to $\uv{\mc I}$. Thus, for this special case, the petit representation provides a faithful Heyting functor
\[ \Hey[\mc I] \inj \mb{opDisc}(\uv{\mc I}). \]

At last, we look at the base change of petit representations, which explains its relationship with the gros representation in more concrete terms:

\begin{lemma}\label{lem:basechangepetit}
  For any Heyting functor $\alpha : \mc H \to \mc G$ in $\HPfp$, the following diagramme commutes,
  \[
    \begin{tikzcd}
      \mc H \ar[d, "\alpha"'] \ar[r, tail] & \opDisc(\Phi_{\mc H}) \ar[d, "\Phi_{\alpha}^{*}"] \\
      \mc G \ar[r, tail] & \opDisc(\Phi_{\mc G})
    \end{tikzcd}
  \]
  where $\Phi_{\alpha} : \Phi_{\mc G} \to \Phi_{\mc H}$ is the morphism of stacks induced by the gros representation, and $\Phi_{\alpha}^{*}$ is the functor of pullback along $\Phi_{\alpha}$, which preserves discrete opfibrations.
\end{lemma}
\begin{proof}
  To show this, it suffices to prove that for any $\varphi\in\mc H$, the following diagramme is a pullback in $\St(\Casfp\iso,J)$,
  \[
    \xymatrix{
      \Phi_{\mc G/\alpha\varphi} \ar[r] \ar[d] & \Phi_{\mc H/\varphi} \ar[d] \\
      \Phi_{\mc G} \ar[r]_{\Phi_{\alpha}} & \Phi_{\mc H}
    }
  \]
  We may dicretely compute the pullback on $\mc P\in\Casfp\iso$ as follows,
  \begin{align*}
    &\,(\Phi_{\mc G} \times_{\Phi_{\mc H}} \Phi_{\mc H/\varphi})(\mc P) \\
    \cong &\, \scomp{G\in\Phi_{\mc G}(\mc P),(F,a) \in \Phi_{\mc H/\varphi}(\mc P)}{G\alpha = F} \\
    \cong & \set{G\in\Phi_{\mc G}(\mc P),a\in(G\alpha\varphi)_{*}} \\
    \cong &\, \Phi_{\mc G/\alpha\varphi}(\mc P)
  \end{align*}
  The first holds since pullback between discrete opfibrations are computed point-wise in $\St(\Casfp\iso,J)$, and the remaining is standard calculation. This completes the proof.
\end{proof}

Now for those Heyting functors induced by pullback along some $\varphi\in\mc H$, the induced morphism between stacks by Lemma~\Ref{lem:petitdisc} will be a discrete opfibration
\[ \Phi_{\varphi} : \Phi_{\mc H/\varphi} \to \Phi_{\mc H}. \]
This means the pullback functor
\[ \Phi_{\varphi}^{*} : \opDisc(\Phi_{\mc H}) \to \opDisc(\Phi_{\mc H/\varphi}) \]
also has a \emph{left adjoint} $\Sigma_{\Phi_{\varphi}}$, which is simply given by composition with $\Phi_{\varphi}$ because discrete opfibrations are closed under composition. In such a case, the petit representation commutes with both adjoints:

\begin{corollary}\label{cor:etalebasechangepetit}
  For any $\varphi\in\mc H$, the following diagramme commutes,
  \[
    \begin{tikzcd}
      {\mc H/\varphi} & {\opDisc(\Phi_{\mc H/\varphi})} \\
      {\mc H} & {\opDisc(\Phi_{\mc H})}
      \arrow[tail, from=1-1, to=1-2]
      \arrow[tail, from=2-1, to=2-2]
      \arrow["\varphi^{*}"', curve={height=6pt}, from=2-1, to=1-1]
      \arrow["\dashv", "\Sigma_{\varphi}"', curve={height=6pt}, from=1-1, to=2-1]
      \arrow["\dashv", "\Sigma_{\Phi_{\varphi}}"', curve={height=6pt}, from=1-2, to=2-2]
      \arrow["\Phi_{\varphi}^{*}"', curve={height=6pt}, from=2-2, to=1-2]
    \end{tikzcd}
  \]
\end{corollary}
\begin{proof}
  From Lemma~\Ref{lem:basechangepetit} we already know it commutes with pullbacks. It commutes with left adjoints by functoriality of gros representation.
\end{proof}

\bibliographystyle{apalike} 
\bibliography{mybib}

\end{document}